\documentclass[12pt]{amsart} 
\usepackage[ansinew]{inputenc}
\usepackage{fancyhdr}
\usepackage{epsfig}
\usepackage{epic}
\usepackage{eepic}
\usepackage{amsfonts}
\usepackage{threeparttable}
\usepackage{amscd}
\usepackage{here}
\usepackage{graphicx}
\usepackage{lscape}
\usepackage{hyperref}
\usepackage{tabularx}
\usepackage{subfigure}
\usepackage{longtable}
\usepackage[pctex32]{color}
\usepackage{colortbl}
\usepackage{amsmath,amsthm}
\usepackage{amssymb,latexsym}
\usepackage{enumerate}
\usepackage{mathrsfs}
\usepackage{layout}
\usepackage[none]{hyphenat}
\usepackage{rotating} 

\newtheorem{lemma}{Lemma} 
\newtheorem{teor}{\sc {\textbf {Theorem}}}
\newtheorem{defin}{\sc {\textbf {Definition}}}
\newtheorem{coro}{\sc {\textbf {Corollary}}}

\newcommand{\F}{\mathcal{F}}
\newcommand{\Df}{\mathsf{F}}

\title[Sectional-Hyperbolic Lyapunov Stable Sets]{Sectional-Hyperbolic Lyapunov Stable Sets} 
\author[Y. Sánchez, S. Bautista. ]{ Y. Sánchez, S. Bautista. }

\address{Y. S\'{a}nchez \\
Departamento de Matem\'{a}ticas \\ 
Universidad Nacional de Colombia, Bogot\'{a}, Colombia} 
\email{yasanchezr@unal.edu.co}
\address{S. Bautista \\
Departamento de Matem\'{a}ticas \\ 
Universidad Nacional de Colombia, Bogot\'{a}, Colombia} 
\email{sbautistad@unal.edu.co}

\date{\today} 


\begin{document} 

\sloppy 

\begin{abstract}

In hyperbolic dynamics, a well-known result is: every hyperbolic Lyapunov stable set, is attracting; it's natural to wonder if this result is maintained in the sectional-hyperbolic dynamics. This question is still open, although some partial results have been presented. We will prove that all sectional-hyperbolic transitive Lyapunov stable set of codimension one of a vector field $X$ over a compact manifold, with unique singularity Lorenz-like, which is of boundary-type, is an attractor of $X$
\end{abstract}

\maketitle 

\section{Introduction}

In the theory of hyperbolic sets we have properties established for both flows and diffeomorphisms. Classical examples of these sets is Smale's horseshoe and the suspension of diffeomorphisms that exhibit hyperbolic sets. After appear a more general class of sets called sectional-hyperbolic sets containing the hyperbolic sets and non-hyperbolic sets like the geometric Lorenz attractor \cite{bm3}. Then a natural problem is to ask what results of the hyperbolic theory are still valid in the sectional-hyperbolic theory. The isolated non-trivial hyperbolic sets have periodic orbits by the shadowing lemma \cite{kh}, but not every sectional-hyperbolic set has periodic orbits as the cherry flow in the torus with a strong contraction (See \cite{bm} page 27). It was recently proved that the sectional-hyperbolic Lyapunov stable sets contain a non-trivial homoclinic class \cite{aml}.\\
 
On the other hand, we know that every hyperbolic Lyapunov stable set is an attracting. The proof of this, is based on the fact that the points in a hyperbolic set have an unstable manifold. In the sectional-hyperbolic set we do not have that unstable manifold guaranteed, except in its singularities and periodic orbits. In this paper we prove that every limit sectional-hyperbolic Lyapunov stable set of codimension one with a unique singularity Lorenz-like of boundary-type, is attracting, for this, we need to use particular case of the connecting lemma for sectional-hyperbolic sets in codimension one presented in \cite{bss}. Below we will specify the definitions and results over sectional-hyperbolic dynamics, with it we works.\\

Hereafter $M$ will be a compact manifold  possibly with nonempty boundary endowed  with a Riemannian metric $\langle\cdot,\cdot\rangle$ an induced norm $||\cdot||$. Given $X$ an $C^1$ vector field, inwardly transverse to the boundary (if nonempty) we call $X_t$ its induced {\em flow} over $M$. Define the {\em maximal invariant set} of $X$
$$
M(X)=\displaystyle\bigcap_{t\geq0}X_t(M).
$$
The \textit{orbit} of a point $p \in M(X)$ is defined by $\mathcal{O}(p)=\{X_t(p)\,|\,t\in\mathbb{R}\}$.
A {\em singularity} will be a zero $q$ of $X$, i.e. $X(q)=0$ (or equivalently $\mathcal{O}(q)=\{q\}$)
and a {\em periodic orbit} is an orbit $\mathcal{O}(p)$ such that $X_T(p)=p$ for some minimal $T>0$ and $\mathcal{O}(p)\neq \{p\}$. By a {\em closed orbit}
we mean a singularity or a periodic orbit.\\

Given $p\in M$ we define the {\em omega-limit set},
$\omega_X(p)=\{x\in M\,|\,x=\lim_{n\rightarrow\infty}X_{t_n}(p)$ for some sequence $t_n\rightarrow\infty\}$, if $p \in M(X)$, define the {\em alpha-limit set}
$\alpha_X(p)=\{x\in M:x=\lim_{n\to{\infty}} X_{-t_n}(p),
$ for some sequence $ t_n\to \infty\}$.\\

A compact subset $\Lambda$ of $M$ is called {\em invariant} if
$X_t(\Lambda)=\Lambda$ for all $t\in\mathbb{R}$; {\em transitive} if $\Lambda = \omega_X(p)$ for some $p\in \Lambda$. A compact invariant set $\Lambda$ is {\em attracting} if there is a neighborhood $U$ such that
\[\Lambda=\cap_{t \geq 0}X_t(U),\] 
and is {\em attractor} of $X$, if is an attracting set $\Lambda$  which is {\em transitive}. On the other hand, a compact invariant set $\Lambda$ is {\em Lyapunov stable}, if for all neighborhood $U$ of $\Lambda$, exists a neighborhood $W$ such that: $X_t(p)\in U$ for every $t\geq 0$ and $p\in W$. 

\begin{defin}
A compact invariant set $\Lambda \subseteq M(X)$ is {\em hyperbolic} if there are positive constants $K,\lambda$and a continuous $DX_t$-invariant splitting of tangent bundle  $T_{\Lambda}M=E^s_{\Lambda}\oplus E^X_{\Lambda}\oplus E^u_{\Lambda}$,  
 such that for every $x \in \Lambda$ and $t \geq 0$:
\begin{enumerate}
\item [$(1)$] $\| DX_t(x)v^s_x \| \leq K e^{-\lambda t}\| v^s_x \|,\ \ \forall v^s_x \in E^s_x$; 
\item [$(2)$] $ \| DX_t(x)v^u_x \| \geq K^{-1} e^{\lambda t} \| v^u_x \|,\ \ \forall v^u_x \in E^u_x$; 
\item [$(3)$] $E^{X}_{x}=\left\langle X(x)\right\rangle $.
\end{enumerate}
\end{defin}

If $E^s_x\neq 0$ and $E^u_x\neq 0$ for all $x\in \Lambda$ we will say that $\Lambda$
is a {\em saddle-type hyperbolic set}. A closed orbit is hyperbolic if it does as a compact invariant set of $X$.\\

The invariant manifold theory \cite{hps} asserts that if $H\subseteq M$ is hyperbolic set of $X$ and $p \in H$, then the topological sets:
$$W^{ss}(p)=\{q \in M: \lim_{t\to\infty} d(X_t(q),X_t(p))=0\}$$
y
$$W^{uu}(p)=\{q \in M(X): \lim_{t\to- \infty} d(X_t(q),X_t(p))=0\}$$

they are $C^1$ manifolds in $M$, so-called strong stable and unstable manifolds, tangent at $p$ to the subbundles
$E^s_p$ and $E^u_p$ respectively. Saturating them with the flow we obtain the stable and unstable 
manifolds $W^{s}(p)$ and $W^{u}(p)$ respectively, which are invariant. If $p,p' \in H$, we have to  
$W^{ss}(p)$ and $W^{ss}(p')$ are same or disjoint (similarly for $W^{uu}$).\\

An {\em homoclinic class} $H_X(p)$ associated to a hyperbolic periodic point $p$ of $X$ is the closure of the transverse intersections between $W^s_X(p)$ and $W^u_X(p)$, i.e.
\[H_X(p)=CL(W^u_X(p)\pitchfork W^s_X(p)).\]
We say that $H\subset M$ is a {\em homoclinic class} of $X$ if $H=H_X(p)$ for some hyperbolic periodic point $p$ of $X$.

\begin{defin}
A compact invariant set $\Lambda \subseteq M(X)$ is {\em sectional-hyperbolic} if every singularity in $\Lambda$ is hyperbolic (as invariant set) and there are a continuous $DX_t$-invariant splitting of tangent bundle $T_{\Lambda} M = \Df^s_{\Lambda}\oplus \Df^c_{\Lambda}$, and positive constants $K,\lambda$
such that for every $x \in \Lambda$ and $t \geq 0$:

\begin{enumerate}
\item
$\| DX_t(x)v^s_x \| \leq K e^{-\lambda t}\| v^s_x \| ,\ \ \forall v^s_x \in \mathsf{F}^s_x$;
\item
$\| DX_t(x)v^s_x \| \cdot \| v^c_x \|  \leq K e^{-\lambda t} 
\| DX_t(x)v^c_x \| \cdot \| v^s_x \|,\ \ \forall v^s_x \in \mathsf{F}^s_x,\ \ \forall v^c_x \in \mathsf{F}^c_x$;
\item
$ \| DX_t(x)u^c_x , DX_t(x)v^c_x  \|_{X_t(x)} \geq K^{-1} e^{\lambda t}  \| u^c_x , v^c_x \|_x, \ \ 
\forall u^c_x , v^c_x \in \mathsf{F}^c_x$.
Where $||\cdot, \cdot ||_x$ it is induced $2$-norm by the Riemannian metrics 
$\langle \cdot, \cdot \rangle_x$ of $T_x\Lambda$, given by
$$||v_x,u_x||_x=\sqrt{\langle v_x,v_x \rangle_x \cdot \langle u_x,u_x\rangle_x - \langle v_x,u_x 
\rangle_x^2}$$ 
for all $x \in \Lambda$ and every $u_x,v_x\in T_x\Lambda$ 
\end{enumerate}
\end{defin}

The third condition guarantees, the increase exponential of the area of parallelograms in the central subbundle $\mathsf{F}^c$. Since $ X(x) \in \Df^c_x$ for all $x \in \Lambda$ (see lemma 4 in \ cite {bm}), we will require that the dimension of the central subbundle must be greater than or equal to $2$. In the particular case where the $ dim (\mathsf{F}^c_x) = 2$ we will say that $\Lambda$ is a sectional-hyperbolic set of \textit{codimension $1$}. \\

Also the invariant manifold theory \cite{hps} asserts that through any point $x$ of a sectional-hyperbolic set $\Lambda$ it passes a 
strong stable manifolds $\F^{ss}(x)$, tangent at $x$ to the subbundle $\Df^s_x$, which induces an foliation over $\Lambda$; saturating them with the flow we obtain the invariant manifold $\F^s(x)$.\\

Unlike hyperbolic sets, the sectional-hyperbolic sets can have regular orbits accumulating singularities. We have to:

\begin{lemma}\label{lema1}
If $\Lambda \subseteq M(X)$ is sectional-hyperbolic set, and $\sigma$ is an singularity in $\Lambda$ then:
$$\F^{ss}(\sigma) \cap \Lambda = \{\sigma\}$$
\end{lemma}

\begin{proof}
See corollary 2 in \cite{bm}.
\end{proof}

All singularity $\sigma$ in an sectional-hyperbolic set, is hyperbolic, so your invariant manifolds $W^{uu}(\sigma)$ and $W^{ss}(\sigma)$ are well defined. The strong  stable manifold sectional $\F^{ss}(\sigma)$ it is a submanifold of $W^{ss}(\sigma)$, with respect to your dimension, exists two possibilities:

\begin{enumerate}
\item $dim(W^{ss}(\sigma)) = dim(\F^{ss} (\sigma))$, in this case $W^{ss}(\sigma) = \F^{ss} (\sigma)$;
\item $dim(W^{ss}(\sigma)) = dim(F^{ss}(\sigma)) + 1$, in this case, we say that the singularity is {\em Lorenz-like}.
\end{enumerate}

Every singularity Lorenz-like is type-saddle hyperbolic set with at least two negative eigenvalues, one of which is real eigenvalue $\lambda$ with multiplicity one such that the real part of the other eigenvalues are outside the closed interval
$[\lambda, -\lambda]$.\\ 

Over a Lorenz-like singularity $\sigma \in \Lambda$, we have $F^{ss}(\sigma)$ is tangent to the subspace associated the eigenvalues with real part less than $\lambda$, and $\F^{ss}(\sigma)$ divide a $W^{ss}(\sigma)$ in two connected component. If $\Lambda$ intersect just one connected component of $W^{ss}(\sigma) \setminus \F^{ss}(\sigma)$, we say that the singularity Lorenz Like is {\em of boundary-type}.\\

On the other hand, we have to:

\begin{lemma}
Let $\Lambda \subseteq M(X)$ a sectional-hyperbolic set and $\sigma$ an singularity it's not Lorenz-like in $\Lambda$. If exists an sequence $x_n\in \Lambda$ of regular points, such that $x_n\to \sigma$, then $x_n \in W^{uu}(\sigma)$ for $n$ large enough.  
\end{lemma}

\begin{proof}
As $\sigma$ it's not Lorenz-like, we have to $\F^{ss}(\sigma)=W^{ss}(\sigma)$ and by lemma \ref{lema1}, $W^{ss}(\sigma) \cap \Lambda =\{\sigma\}$, then $x_n \notin W^{ss}(\sigma)$, for all $n$. Suppose that exists an subsequence $X_{n_k}$ such that $x_{n_k} \notin W^{uu}(\sigma)$, then exists a regular point $y \in M$, such that:
 $$y \in Cl \left( \bigcup_{k=1}^\infty \mathcal{O}^-(x_{n_k}) \right) \cap  (W^{ss}(\sigma)\setminus \{\sigma\})$$
 but as $x_{n_k}\in \Lambda$, so by the compactness of $\Lambda$, $y \in \Lambda \cap W^{ss}(\sigma) \setminus \{\sigma\}$, this is a contradiction, thus $x_n \in W^{uu}(\sigma)$ for $n$ large enough. 
\end{proof}

Then, the only singularities that can be accumulated by positive orbits of regular points in a sectional-hyperbolic set, are Lorenz-like.

\begin{coro}\label{coro2}
If $\Lambda \subseteq M(X)$ is a sectional-hyperbolic set, $\Lambda$ it's not a singularity and exists $q\in M$ such that $\Lambda=\alpha(q)$ or $\Lambda=\omega(q)$, then every singularity in $\Lambda$ is Lorenz-like.
\end{coro}

We say that a cross section $\Sigma$ of $X$ is associated to a Lorenz-like singularity $\sigma$ in a sectional-hyperbolic set $\Lambda$, if $\Sigma$ is very close to $\sigma$, $\Sigma \cap \Lambda \neq \emptyset$ and one of the connected components of $W^{ss}(\sigma) \setminus \F^{ss}(\sigma)$ contains a point in $int(\Sigma)$.\\

Another important result about the sectional-hyperbolic sets, is the \emph{\textbf{hyperbolic lemma}} (see lemma 9 in \cite{bm}), which assert that any invariant subset $H$ without singularities of a sectional-hyperbolic set $\Lambda$,  is hyperbolic, in this case, we have to that $\Df^s_H=E^s_H$ and $\Df^c_H=E^u_H \oplus E^X_H$, so $W^{ss}(p)=\F^{ss}(p)$ for all $p \in H$.\\

Observe that the closed orbits of a sectional-hyperbolic set always have a hyperbolic structure, the periodic orbits by the hyperbolic lemma and the singularities by definition. the follow two definitions apply to sets whose closed orbits are hyperbolic:

\begin{defin}
A compact invariant set $\Lambda$ of a vector field $X$ over $M$, has the property $(P)$ if for all periodic orbit $O$ in $\Lambda$ exists an singularity $\sigma$ also in  $\Lambda$ such that $W^u(O) \cap W^s(\sigma)\neq \emptyset$.
\end{defin}

\begin{defin}
A sectional-hyperbolic set $\Lambda \in M(X)$, has the property $(S)$ if for all Lorenz-like singularity  $\sigma$ in $\Lambda$ exists an periodic orbit $p$ also in  $\Lambda$ such that $W^u(O) \cap W^s(\sigma)\neq \emptyset$.
\end{defin}

\section{Lyapunov Stable Attracting Sets}

In this section will establish some sufficient conditions for that a sectional-hyperbolic Lyapunov stable  set to be attracting.

\begin{lemma}\label{lema3}
Let $\Lambda$ a Lyapunov stable set of a vector field $X$ over $M$ then: 

\begin{enumerate}
\item $\{y \in M : d(X_{-t}(x),X_{-t}(y))\to 0$ when $t\to \infty \} \subseteq \Lambda$, $\forall x \in \Lambda$.
\item $\Lambda$ is a attracting of $X$ if and only if there is a neighborhood $U$ of $\Lambda$ such that $\omega_X(x)\subseteq \Lambda$ for all $x\in U$.
\end{enumerate}
\end{lemma}

\begin{proof}
See lemma 2.25 pag 35 and lemma 2.26 pag 36 en \cite{ap}.
\end{proof}

\begin{teor}\label{teor1}
Every sectional-hyperbolic Lyapunov stable set $\Lambda$ of a vector field $X$ over $M$, that satisfies the property $(S)$ and whose Lorenz-like singularities are of boundary type, is an attracting of $X$.
\end{teor}

\begin{proof}
Denote by $T_\Lambda M = \Df^s_\Lambda \oplus \Df^c_\Lambda$ the sectional splitting of $\Lambda$, which we can extends a $T_U M= \Df^s_U \oplus \Df^c_U$ where  $U$ is a neighborhood of $\Lambda$ in $M$; this extension is continuous for $\Df^c_U$ and integrable for $\Df^s_U$. In that follows we will hold the neighborhood $U$ of $\Lambda$.\\

To prove that $\Lambda$ is an attracting, we will use the item (2) of the lemma \ref{lema1}, so, it suffices to prove that: if $x_n \in M$ is a sequence converging to $p \in \Lambda$, then $\omega_X (x_n) \subseteq \Lambda$ for $n$ large. We have two possible cases:\\

\begin{enumerate}
\item[{\bf Case 1:}] {\it $Sing(X) \cap \omega_X(p)=\emptyset$}. So, by the hyperbolic lemma, $\omega_X(p)$ is a saddle-type hyperbolic set. Choose $y \in \omega_X(p)$, then $W^{uu}(y)$ is well defined and by the item (1) of the lemma \ref{lema1} we have that $W^{uu}(y)\subseteq \Lambda$. Let $\Sigma \subseteq U$  a cross section of $X$ with $y \in int(\Sigma)$. Denote by $\F^s_{\Sigma}$ the {\em vertical foliation} of $\Sigma$ obtained by projecting $\F^{ss}$ into $\Sigma$ along the flow of $X$, (i.e. $\F^s(x,\Sigma)$ is the leaf in $\Sigma$ obtained by projecting the leaf $\F^{ss}(x)$ in $\Sigma$ along the flow of $X$, for all $x \in \Sigma$). Choose $\Sigma$ small in size, we have that $W^{uu}(y) \cap \F^s(x,\Sigma) \neq \emptyset$ for all $x\in \Sigma$.\\

As $y \in \omega_X(p)$ we have that the positive orbit of $p$ intersect a $\Sigma$, this implies that, for $n$ large, the positive orbit of $x_n$ also does in a point $x'_n$, since $x_n \to p$; then, there exists $z_n \in \F^s(x'_n,\Sigma) \cap W^{uu}(p) \subseteq \Lambda$, such that $x'_n \in \F^{ss}(z_n)$. So, $\omega_X(x_n) = \omega_X(z_n) \subseteq \Lambda$.\\

\item[{\bf Case 2:}] {\it $Sing(X) \cap \omega_X(p) \neq \emptyset$}. Let $\sigma \in \omega_X(p)$, then $\sigma$ it is accumulated by the positive orbit of $p \in \Lambda$ and therefore is Lorenz-Like singularity. By the property ($S$), exists a hyperbolic periodic orbit $O \in \Lambda$ such that $z \in W^u(O) \cap W^s(\sigma)$ for some point $z\in M$. As $\Lambda$ is a Lyapunov stable set, then $z\in W^u(O) \subseteq \Lambda$. Since $\F^{ss}(\sigma) \cap \Lambda = \{ \sigma \}$, so $z \in W^{ss}(\sigma)\setminus \F^{ss}(\sigma)$.\\ 

Let $\Sigma \subset U$ a cross section associated a $\sigma$, as this it is singularity of boundary-type, exists $z' \in Int(\Sigma) \cap \mathcal{O}(z)$. Denote by $\F^s_\Sigma$ the {\em vertical foliation} of $\Sigma$ and by $\partial^v \Sigma$ and $\partial^h \Sigma$ the vertical and horizontal boundary of $\Sigma$ respectively. We will assume that the vertical boundary $\partial^v \Sigma$ is formed by leaves of the foliation $\F^s_\Sigma$, and $\partial^h \Sigma$ is transverse to $\F^s_\Sigma$. Given that $\F^{ss}(\sigma) \cap \Lambda = \{ \sigma \}$, $\Sigma$ can be choose such that  $\partial^h \Sigma \cap \Lambda = \emptyset$. In addition, every orbit that accumulate to $\sigma$ necessarily, accumulate the  leaf $\F^s(z',\Sigma)$ in $\Sigma$ .\\

Now, since $z'\in int(\Sigma )\cap W^{u}(O) \cap W^s(\sigma) \subseteq \Lambda$, then we can choose $\Sigma$ small in size such that: $$\F^s(x,\Sigma) \cap \Lambda \ne \emptyset,$$ for all $x \in int(\Sigma)$ next to $\F^s(z',\Sigma)$. As $\sigma \in \omega_X(p)$ and $\sigma$ is of boundary-type, necessarily the  positive orbit of $p$ intersects to $\F^s(z',\Sigma)$ or intersects infinite times to $\Sigma$ accumulating to  $\F^s(z',\Sigma)$, then for $n$ large enough we have the positive orbit of $x_n$ intersects to $int(\Sigma)$, next to $\F^s(z',\Sigma)$. From which it is concluded that, for  $n$ large, exists $z_n \in \Lambda$ such that $x_n \in \F^{s}(z_n)$. Then, $\omega_X(x_n) = \omega_X(z_n) \subseteq \Lambda$.
\end{enumerate}
\end{proof}

\begin{coro}
Every sectional-hyperbolic Lyapunov stable set $\Lambda$ of a vector field $X$ over $M$, without Lorenz-like singularities, is an attracting of $X$.
\end{coro}

\begin{coro}\label{coro1}
Every sectional-hyperbolic Lyapunov stable set $\Lambda$ of a vector field $X$ over $M$, that satisfies the property $(P)$ with a unique singularity Lorenz-like, which is of boundary type, is an attracting of $X$.
\end{coro}

\begin{proof}
Let $\sigma$, the only Lorenz-like singularity in $\Lambda$. By theorem 1.1 in \cite{aml}, $\Lambda$ has a nontrivial homoclinic class, and therefore a periodic point $q$. Since $\Lambda$ satisfies the property $(P)$, then there exists $\sigma^* \in \Lambda \cap Sing(X)$ such that $W^u(q)\cap W^s(\sigma^*)\neq \emptyset$, as $W^u(q) \subseteq \Lambda$ for being $\Lambda$ Lyapunov stable, we have that $\sigma^*$ is Lorenz-Like, so $\sigma^*=\sigma$, then $\Lambda$ satisfies the property $(S)$ and by the theorem \ref{teor1}, $\Lambda$ is an attracting.
\end{proof}

In the case that the Lyapunov stable sectional-hyperbolic set is of codimension one, we can replace, the hypothesis that the property $(P)$ is fulfilled in the corollary \ref{coro1}, by transitivity. Let $p,q \in M$ we will say that $p \prec q$ if for all $\epsilon>0$ there is a trajectory from a point $\epsilon$-close to $p$ to a point $\epsilon$-close to $q$. 

\begin{teor}[\textbf{Main}]
Let sectional-hyperbolic Lyapunov stable set $\Lambda$ of a vector field $X$ over $M$ of codimension 1, with a unique singularity Lorenz-like, which is of boundary type. If $\Lambda=\omega_X(q)$  $(\text{or } \Lambda = \alpha_X(q))$ for some point $q\in M$, then $\Lambda$ is an attracting of $X$.
\end{teor}

\begin{proof}
Let us verify that under the assumptions of the theorem, $\Lambda$ satisfies the property $(S)$. Let $\sigma$ the only Lorenz-like singularity in $\Lambda$, by theorem 1.1 in \cite{aml}, $\Lambda$ has a nontrivial homoclinic class, and therefore a periodic point $p$. If there exists $q\in M$ such that   $\alpha_X(q)=\Lambda$ or $\omega_X(q)=\Lambda$, by the corollary \ref{coro2}, every singularity in $\Lambda$ is Lorenz-like, then $\sigma$ is the unique singularity of $\Lambda$. Also using the orbit of $q$ we have that $p \prec \sigma$, and since $\Lambda$ satisfies the conditions of the theorem 10 in \cite{bss}, so there exists $x\in \Lambda$ such that $\alpha_X(x)=\alpha_X(p)$ and $\omega_X(x)$ is a singularity, and this case, necessarily, $\omega_X(x) =\{\sigma\}$; that is,  $W^u(p)\cap W^s(\sigma) \ne \emptyset$. Then by theorem \ref{teor1}, $\Lambda$ is an attracting of $X$. 
\end{proof}

As direct consequence of the main theorem we have that: 

\begin{coro}
Let sectional-hyperbolic transitive Lyapunov stable set $\Lambda$ of a vector field $X$ over $M$ of codimension 1, with a unique singularity Lorenz-like, which is of boundary type, then $\Lambda$ is an attractor of $X$.
\end{coro}

\end{document}